\documentclass[10pt, a4paper, reqno, english]{amsart}

\usepackage{hyperref}
\usepackage{amssymb}
\usepackage[all]{xy}

\newcommand\ZZ{\mathbb{Z}}
\newcommand\CC{\mathbb{C}}
\newcommand\RR{\mathbb{R}}
\newcommand\PP{\mathbb{P}}
\newcommand\QQ{\mathbb{Q}}
\newcommand\xx{\mathbf{x}}
\newcommand{\Aone}{{\mathbf A}_1}
\newcommand{\Atwo}{{\mathbf A}_2}

\newcommand{\Afive}{{\mathbf A}_5}
\newcommand{\Dfour}{{\mathbf D}_4}
\newcommand{\Dfive}{{\mathbf D}_5}
\newcommand{\Esix}{{\mathbf E}_6}
\newcommand{\tS}{{\widetilde S}}
\newcommand{\abs}[1]{\left\|#1\right\|_\infty}
\newcommand{\ee}{\boldsymbol{\eta}}
\newcommand{\e}{\eta}
\newcommand\dd{\,\mathrm{d}}
\newcommand{\ex}[1]{*+<5pt>[o][F]{E_{#1}}}
\newcommand{\li}[1]{*+<3pt>[F]{E_{#1}}}
\newcommand{\classrep}{\mathcal{C}}
\newcommand{\CCC}{\mathcal{C}}
\newcommand{\classtuple}{\mathbf{C}}
\newcommand{\OO}{\mathcal{O}}
\newcommand{\eI}{I}
\newcommand{\eII}{\mathbf{I}}
\newcommand{\N}{\mathfrak{N}}
\newcommand{\id}[1]{\mathfrak{#1}}
\newcommand{\aaa}{\id a}
\newcommand{\bbb}{\id b}
\newcommand{\ddd}{\id d}
\newcommand{\qqq}{\id q}
\newcommand{\DD}{\mathcal{D}}
\newcommand{\II}{\mathcal{I}}
\newcommand{\p}{\id p}
\newcommand{\kc}{\id{k}_\id{c}}
\newcommand{\Ao}{\Pi_1}
\newcommand{\At}{\Pi_2}

\newcommand\rto{\dashrightarrow}
\newcommand{\dual}[1]{#1^*}
\newcommand{\dualtr}[1]{#1^\vee}
\newcommand{\ip}[2]{\langle #1,#2\rangle}
\newcommand{\ft}[1]{\widehat{#1}}
\newcommand{\numunits}{|\OO_K^\times|}

\DeclareMathOperator{\Tr}{Tr}

\newtheorem{theorem}{Theorem}
\newtheorem{lemma}[theorem]{Lemma}
\newtheorem{prop}[theorem]{Proposition}
\theoremstyle{definition}
\newtheorem*{ack}{Acknowledgements}

\numberwithin{theorem}{section}
\numberwithin{equation}{section}

\begin{document}

\setcounter{tocdepth}{1}

\title[On Manin's conjecture for a certain singular cubic surface]{On Manin's
  conjecture for a certain singular cubic surface over imaginary quadratic
  fields}

\author{Ulrich Derenthal}

\address{Mathematisches Institut, Ludwig-Maximilians-Universit\"at M\"unchen,
  Theresienstr. 39, 80333 M\"unchen, Germany}

\email{ulrich.derenthal@mathematik.uni-muenchen.de}

\author{Christopher Frei}

\email{frei@math.lmu.de}

\date{January 22, 2014}

\begin{abstract}
  We prove Manin's conjecture over imaginary quadratic number fields
  for a cubic surface with a singularity of type $\Esix$.
\end{abstract}


%

\maketitle

\tableofcontents

\section{Introduction}

Central questions in the arithmetic of cubic surfaces over number fields are
the existence and distribution of rational points on them. It is known that the
Hasse principle holds for singular cubic surfaces over number fields
\cite{MR0073643}, but may fail for smooth cubic surfaces over $\QQ$ \cite[\S
2]{MR0139989}. Weak approximation may fail for smooth and singular cubic
surfaces; see \cite[\S 3]{MR0139989} for the first singular example over
$\QQ$. Failures of the Hasse principle and weak approximation are explained by
Brauer--Manin obstructions \cite{MR0427322} in all known examples, but it
remains open whether the Brauer--Manin obstruction to the Hasse principle and
weak approximation is the only one for all cubic surfaces.

Manin's conjecture makes a precise prediction of the quantitative
behavior of rational points on cubic surfaces and, more generally,
Fano varieties. In fact, numerical experiments for a cubic surface
\cite[Appendix]{MR89m:11060} played a central role in Manin's first
formulation of this conjecture. For a cubic surface $S \subset
\PP^3_K$ containing at least one rational point over a number field
$K$, Manin's conjecture predicts that the number of rational points of
Weil height $H(\xx)$ bounded by $B$
\begin{equation*}
  N_{U,H}(B) := |\{\xx \in U(K) \mid H(\xx) \leq B\}|
\end{equation*}
on the complement $U$ of the lines in $S$ behaves asymptotically, for $B \to
\infty$, as
\begin{equation*}
  N_{U,H}(B) = c_{S, H} B(\log B)^{\rho-1}(1+o(1)),
\end{equation*}
where $\rho$ is the rank of the Picard group of (the minimal desingularization
of) $S$, and the leading constant $c_{S, H} > 0$ has an explicit
interpretation due to Peyre \cite{MR1340296, MR2019019} and
Batyrev--Tschinkel \cite{MR1679843}.

Central techniques to prove Manin's conjecture are harmonic analysis
in case of varieties that are equivariant compactifications of
algebraic groups \cite{MR1620682, MR1906155}, the circle method in
case of varieties defined by forms in many variables \cite{MR0150129,
  MR1340296}, and universal torsors \cite{MR89f:11082, MR1679841} for
varieties without such a special structure. For cubic surfaces, one
may alternatively use conic fibrations to parameterize rational
points.

For many years the only cubic surfaces for which Manin's conjecture could be
proved were forms of the cubic surface defined by the
equation
\begin{equation}\label{eq:def_3A2}
  x_0^3=x_1x_2x_3,
\end{equation}
of singularity type $3\Atwo$ (as special cases of \cite{MR1620682,
 MR1679841}; see \cite{MR3107569} for references to articles giving
other proofs of Manin's conjecture for this particular surface). The reason
is that it is toric, so that on the one hand, it is accessible to the harmonic
analysis method, and on the other hand, its universal torsor is particularly
simple \cite{MR95i:14046}.

However, no other cubic surface is an equivariant compactification of an
algebraic group \cite{MR2753646, arXiv:1212.3518}. Furthermore, cubic
surfaces are clearly out of reach of the circle method. The conic fibration
approach gave the best available upper bounds $N_{U,H}(B) \ll
B^{4/3+\epsilon}$ for certain smooth cubic surfaces over $\QQ$
\cite{MR98h:11083}. The universal torsor approach led to the proof of Manin's
conjecture over $\QQ$ for the cubic surface defined by
\begin{equation}\label{eq:def_E6}
  x_0^2x_2+x_1x_2^2+x_3^3 = 0
\end{equation}
with an $\Esix$ singularity \cite{MR2332351}, based on the computation
of its Cox ring \cite{MR2029868} and analytic number theory. Proofs of
Manin's conjecture over $\QQ$ via universal torsors for several
further cubic surfaces followed, overcoming new obstacles in each case
(singularity types $\Dfive$ \cite{MR2520769}, $2\Atwo+\Aone$
\cite{MR2990624}, $\Afive+\Aone$ \cite{arXiv:1205.0373}, $\Dfour$
\cite{arXiv:1207.2685}).

Recently, we started to generalize the universal torsor approach from $\QQ$ to
other number fields. This is inspired by the work of Schanuel \cite{MR557080},
which can be interpreted as the proof of Manin's conjecture for projective
spaces over arbitrary number fields via universal torsors.

Our first step was to revisit the toric singular cubic surface defined by
(\ref{eq:def_3A2}) using universal torsor techniques (with Janda
\cite{arXiv:1105.2807} over imaginary quadratic fields of class number $1$,
\cite{MR3107569} over arbitrary number fields).

Our second step was to go beyond toric varieties over imaginary quadratic
fields. The testing ground were certain del Pezzo surfaces of higher degree
(just as over $\QQ$, where the investigation of Manin's conjecture beyond
equivariant compactifications of algebraic groups and forms in many variables
started with smooth quintic \cite{MR1909606} and singular quartic
\cite{MR2373960} del Pezzo surfaces). In \cite{arXiv:1302.6151}, we developed
the necessary techniques over imaginary quadratic fields in some generality
and applied them to a first example, and in \cite{arXiv:1304.3352} we showed
that they apply to some other singular quartic del Pezzo surfaces. While our
general techniques apply in principle to del Pezzo surfaces of arbitrary
degree, they do not provide sufficiently strong bounds for the error terms to
prove Manin's conjecture for any cubic surface.

Our third step is to prove here that Manin's conjecture holds over an
arbitrary imaginary quadratic field $K$ for the cubic surface $S \subset
\PP^3_K$ of type $\Esix$ defined by (\ref{eq:def_E6}),
continuing the investigations from \cite{MR2332351}.

\begin{theorem}\label{thm:main}
  Let $K \subset \CC$ be an imaginary quadratic field.  Let $S$ be the
  cubic surface over $K$ of type $\Esix$ defined by (\ref{eq:def_E6}), and let $U$ be
  the complement of the line $L = \{x_2=x_3=0\}$ in it.  For $B \geq
  3$, we have
  \begin{equation*}
    N_{U,H}(B) = c_{S, H} B(\log B)^6 + O(B(\log B)^5\log \log B),
  \end{equation*}
  where $H$ is the usual exponential Weil height, and
  \begin{equation*}
    c_{S,H} :=\frac{1}{6220800}
    \cdot \frac{(2\pi)^7 h_K^7}{\sqrt{|\Delta_K|}^9\numunits^7}\cdot 
    \prod_\p \left(1-\frac{1}{\N\p}\right)^7\left(1+\frac{7}{\N\p}+\frac{1}{\N\p^2}\right)
    \cdot \omega_{\infty}(\tS)\text.
  \end{equation*}
  Here, $\Delta_K$ is the discriminant, $h_K$ the class number, $\numunits$ the
  number of units in the ring of integers $\OO_K$ of $K$, $\p$ runs over
  all nonzero prime ideals of $\OO_K$, and $\N\p$ is the absolute norm of $\p$, while
  \begin{equation*}
      \omega_\infty(\tS):=\frac{12}{\pi}\int_{\abs{z_0z_1^2},\abs{z_0^2z_1+z_2^3}, \abs{z_1^3}, \abs{z_1^2z_2} \le 1} \dd z_0 \dd z_1 \dd z_2
  \end{equation*}
  is a complex integral, with bounds defined via $\abs{z}=z\overline{z}$ for $z
  \in \CC$.
\end{theorem}

The implied constant in the error term is allowed to depend on $K$. Theorem
\ref{thm:main} agrees with Manin's conjecture since $S$ is split over $K$,
hence its minimal desingularization $\tS$ has a Picard group of rank $7$.

We present the proof of Theorem~\ref{thm:main} in Sections
\ref{sec:passage}--\ref{sec:proof_main}: First, the rational points of
bounded height on $U$ are parameterized by integral points on
universal torsors over $\tS$, subject to some coprimality- and height
conditions. Then, these points on the universal torsors are counted by
means of analytic number theory, in several summations, one for each
coordinate. The main challenge here is the treatment of the error
terms.

One key analytic ingredient is the following variant of the Gau\ss\ circle
problem where, in addition, the center is summed over quadratic residues modulo
an ideal $\qqq$ of $\OO_K$. This is proved in Section \ref{sec:circle_problem}
and may be of independent interest. Let $\omega_K$ and $\phi_K$ denote the
prime divisor function and Euler's $\phi$-function on the nonzero ideals of
$\OO_K$. Moreover, we write $\phi^*_K(\aaa):=\phi_K(\aaa)/\N\aaa$.

\begin{theorem}\label{thm:circle_problem}
  Let $\aaa$, $\qqq$ be nonzero ideals of $\OO_K$, $\alpha \in \OO_K$ with
  $\aaa + \qqq = \alpha\OO_K + \qqq = \OO_K$, and $\epsilon>0$. Then, for $t
  \geq 0$,
  \begin{equation*}
    \sum_{\substack{\rho \bmod \qqq\\\rho\OO_K + \qqq =
        \OO_K}}\sum_{\substack{z \in \aaa\\z \equiv \alpha\rho^2 \bmod
        \qqq\\\abs{z}\leq t}}1 =
    \frac{2\pi\phi_K^*(\qqq)}{\sqrt{|\Delta_K|}\N\aaa}t +
    O_\epsilon\left(\left(\frac{t}{\N\aaa}\right)^{1/3}\N\qqq^{1/3+\epsilon} +
      2^{\omega_K(\qqq)}\N\qqq^{1/2}\right).
  \end{equation*}
 Here, $\rho$ runs over a reduced residue system of $\OO_K$ modulo $\qqq$.
\end{theorem}

Just summing the (naive) error terms of the inner sum over $\rho$ would yield
the total error $O((t\N\qqq/\N\aaa)^{1/2}+\N\qqq)$, which is insufficient for
our applications.  If $K=\QQ(i)$, $\qqq = \aaa = \ZZ[i]$, Theorem
\ref{thm:circle_problem} gives the Gau\ss\ circle problem with Sierpi\'nski's
\cite{Sierpinski1906} classical error term $O(t^{1/3})$ and an additional error
$O(1)$ to take care of small $t$. If $\qqq \neq \ZZ[i]$, $\alpha=1$, Theorem
\ref{thm:circle_problem} can be interpreted as counting only quadratic residues
modulo $\qqq$ in the circle problem.

To obtain a sufficiently strong error term, we use Poisson summation.
Additional difficulties arise from the fact that we do not only need
error cancellation in terms of the circle's radius, but also in terms
of the norm of $\qqq$. To this end, we estimate quadratic exponential
sums over $K$.  This new approach leads, in particular, to a crucial
improvement of the general treatment of the second summation in
\cite[Section 6]{arXiv:1302.6151}; see Section \ref{sec:second_sum}.

\subsection{Notation}
As in \cite{arXiv:1302.6151}, we use the following
notation. Let $\classrep$ be a fixed system of integral
representatives for the ideal classes of the ring of integers
$\OO_K$. The symbol $\p$ always denotes a nonzero prime ideal of
$\OO_K$, and products indexed by $\p$ are understood to run over all
such prime ideals. We say that $x \in K$ is \emph{defined}
(resp. \emph{invertible}) modulo an ideal $\aaa$ of $\OO_K$ if
$v_\p(x) \geq 0$ (resp. $v_\p(x) = 0$) for all $\p \mid \aaa$, where
$v_\p$ is the usual $\p$-adic valuation. For $x,y$ defined modulo
$\aaa$, we write $x \equiv_\aaa y$ if $v_\p(x-y)\geq v_\p(\aaa)$ for
all $\p \mid \aaa$.

We write $\II_K$ for the monoid of nonzero ideals of $\OO_K$ and $\mu_K$ for
the M\"obius function on $\II_K$. For a fractional ideal $\aaa$ of $\OO_K$, we
write $\aaa^{\neq 0} := \aaa \smallsetminus \{0\}$.

The implied constants in Vinogradov's $\ll$- and Landau's $O$-notation may
depend on $K$ and on $\epsilon > 0$. Additional
dependencies are indicated by appropriate subscripts.

\begin{ack}
  The first-named author was supported by grants DE\,1646/2-1 and DE\,1646/3-1
  of the Deutsche Forschungsgemeinschaft. The second-named author was
  partially supported by a research fellowship of the Alexander von Humboldt
  Foundation.
\end{ack}

\section{The circle problem}\label{sec:circle_problem}
In this section, we prove Theorem \ref{thm:circle_problem}. We
identify $\CC$ with $\RR^2$ and use the inner product
$\ip{a+bi}{c+di}:=ac+bd$. For a lattice $\Lambda \subset \CC$, we
denote its dual lattice with respect to $\ip{\cdot}{\cdot}$ by
\begin{equation*}
  \dual\Lambda = \{w \in \CC \mid \ip{v}{w} \in \ZZ \text{ for all }v \in \Lambda\}\text.
\end{equation*}
If $\Lambda \subset K$ then the dual lattice with respect to the trace pairing is denoted by
\begin{equation*}
  \dualtr\Lambda = \{w \in K \mid \Tr_{K|\QQ}(vw) \in \ZZ \text{ for all }v \in \Lambda\}\text.
\end{equation*}
Since $\ip{v}{w} = \Tr(v\overline{w}/2)$, we have $\dual\Lambda = 2
\overline{\dualtr{\Lambda}}$, and if $\Lambda = \qqq$ is a fractional ideal of
$K$ then $\dualtr\qqq$ and $\dual\qqq$ are as well, namely $\dualtr\qqq =
\qqq^{-1}\DD_K^{-1}$, $\dual\qqq = 2 \overline{\qqq^{-1}}\DD_K^{-1}$, where
$\DD_K$ denotes the different of $K$ (over $\QQ$). We will apply the Poisson
summation formula with respect to $\ip{\cdot}{\cdot}$ and estimate exponential sums
with respect to $\Tr(\cdot)$. The above paragraph shows how to translate
between the two pairings. 

The following two lemmas adapt the result of \cite{MR0047697} on
exponential sums over number fields to our needs.

\begin{lemma}\label{lem:exp_sum_1}
  Let $\qqq$ be a nonzero ideal of $\OO_K$ and let $w \in \dualtr\qqq$. Then, for $\epsilon
  > 0$,
  \begin{equation*}
    \sum_{\beta \bmod \qqq}e^{2\pi i \Tr(w\beta^2)} \ll \N(w\qqq\DD_K + \qqq)^{1/2-\epsilon}\N\qqq^{1/2+\epsilon}\text. 
  \end{equation*}
\end{lemma}

\begin{proof}
  Let $\aaa := w (\dualtr{\qqq})^{-1}$ and $\ddd :=
  \aaa+\qqq$. Then all $\beta \in \qqq\ddd^{-1}$ satisfy
  \begin{equation*}
    w \cdot \beta^2 \in \aaa \dualtr{\qqq}\cdot\qqq\ddd^{-1} =
    (\aaa\ddd^{-1})\dualtr{\qqq}\qqq \subset \DD_K^{-1} = \dualtr{\OO_K}\text, 
  \end{equation*}
  so $\Tr(w\beta^2) \in \ZZ$. Therefore, the summand
  depends only on $\beta \bmod \qqq\ddd^{-1}$ and we obtain
  \begin{equation*}
   \sum_{\beta \bmod \qqq}e^{2\pi i \Tr(w\beta^2)}=\N\ddd\sum_{\beta\bmod\qqq\ddd^{-1}}e^{2\pi i\Tr(w\beta^2)}\text.
  \end{equation*}
  Now $w\DD_K = \aaa\qqq^{-1} = \aaa\ddd^{-1}/(\qqq\ddd^{-1})$. Since
  the numerator and denominator of this expression are relatively prime, we can
  apply \cite[Theorem 1]{MR0047697} to obtain the upper bound
  \begin{equation*}
    \ll
    \N\ddd\N(\qqq\ddd^{-1})^{1/2+\epsilon} =
    \N\ddd^{1/2-\epsilon}\N\qqq^{1/2+\epsilon}\text.\qedhere
  \end{equation*}
\end{proof}

\begin{lemma}\label{lem:exp_sum_2}
  Let $\qqq$ be a nonzero ideal of $\OO_K$ and $w \in \dualtr\qqq$. Then, for
  $\epsilon > 0$,
  \begin{equation*}
    \sum_{\substack{\beta \bmod \qqq\\\beta\OO_K + \qqq = \OO_K}}e^{2\pi i
      \Tr(w\beta^2)} \ll \N(w\qqq\DD_K +\qqq)^{1/2-\epsilon}\N\qqq^{1/2+2\epsilon}\text. 
  \end{equation*}
\end{lemma}

  \begin{proof}
    Let
    \begin{equation*}
      W:= \sum_{\substack{\beta \bmod \qqq\\\beta\OO_K + \qqq = \OO_K}}e^{2\pi i\Tr(w\beta^2)}
         = \sum_{\aaa \mid \qqq}\mu_K(\aaa)\sum_{\substack{\beta \bmod
          \qqq\\\beta \in \aaa}}e^{2\pi i\Tr(w\beta^2)}\text.
    \end{equation*}
    Let $x_\aaa \in \aaa$ such that $x_\aaa \aaa^{-1} + \qqq\aaa^{-1} =
    \OO_K$. Then $y+\qqq\aaa^{-1} \mapsto x_\aaa y + \qqq$ defines an isomorphism $\OO_K/(\qqq\aaa^{-1})
    \to \aaa/\qqq$. Hence, using Lemma \ref{lem:exp_sum_1},
    \begin{align*}
      W &= \sum_{\aaa \mid \qqq}\mu_K(\aaa)\sum_{\beta_0
        \bmod \qqq\aaa^{-1}}e^{2\pi i \Tr(w x_\aaa^2 \beta_0^2)}\\
      &\ll \sum_{\aaa\mid\qqq}|\mu_K(\aaa)|\N(w x_\aaa^2
      \qqq\aaa^{-1}\DD_K + \qqq\aaa^{-1})^{1/2-\epsilon}\N(\qqq\aaa^{-1})^{1/2+\epsilon}\text.
    \end{align*}
    Since $\N(w x_\aaa^2\qqq\aaa^{-1}\DD_K +
    \qqq\aaa^{-1})\leq\N\aaa^{-1}\N(x_\aaa^2\OO_K+\qqq)\N(w\qqq\DD_K + \qqq) \leq \N\aaa\N(w\qqq\DD_K+\qqq)$,
    we obtain 
    \begin{equation*}
      W \ll
      \N(w\qqq\DD_K+\qqq)^{1/2-\epsilon}\N\qqq^{1/2+\epsilon}\sum_{\aaa\mid\qqq}|\mu_K(\aaa)|\N\aaa^{-2\epsilon}
      \ll \N(w\qqq\DD_K+\qqq)^{1/2-\epsilon}\N\qqq^{1/2+2\epsilon}\text.\qedhere
    \end{equation*}
  \end{proof}

  \begin{lemma}\label{lem:param_ideal}
    Let $\aaa$ be a nonzero fractional ideal of $K$. Then there is an $\RR$-linear map
    $\varphi : \CC \to \CC$ with $\varphi(\ZZ[i]) = \aaa$ such that for
    all $v \in \RR^2$ we have 
    \begin{equation*}
      \N\aaa \abs{v}\ll\abs{\varphi(v)}\ll\N\aaa\abs{v}.
    \end{equation*}
  \end{lemma}

  \begin{proof}
    By \cite[Lemma VIII.1, Lemma V.8]{MR1434478}, there exists a basis $w_1,
    w_2$ of $\aaa$ with $|w_i| = \lambda_i$, where $\lambda_1\leq \lambda_2$
    are the successive minima of $\aaa$ (with respect to the unit ball). Define
    $\varphi$ by $\varphi(1) = w_1$, $\varphi(i) = w_2$. Clearly, its operator
    norm $|\varphi|$ is bounded by $2\lambda_2$. Together with the inequality
    $|\varphi^{-1}| \leq 2 \lambda_2/\det \aaa$ (see, e.g., the proof of
    \cite[Lemma 3.3]{arXiv:1302.6151}), this gives
    \begin{equation*}
      \left(\frac{\det \aaa}{2\lambda_2}\right)^2\abs{v}  \leq \abs{\varphi(v)} \leq (2 \lambda_2)^2 \abs{v}\text.
    \end{equation*}
    Minkowski's second theorem and the fact that $\lambda_1 \geq \sqrt{\N\aaa}$
    (see, e.g., \cite[Lemma 5]{MR2247898}) imply that $\lambda_2 \ll
    \sqrt{\N\aaa}$ and $\det\aaa/(\lambda_2) \gg \sqrt{\N\aaa}$.
  \end{proof}

  \begin{proof}[Proof of Theorem \ref{thm:circle_problem}]
    Denote the left-hand side by $Z$. If $t < 4\sqrt{\N\qqq}\N\aaa$
    then
   \begin{equation*}
     Z \ll 2^{\omega_K(\qqq)}\sum_{\substack{z \in \aaa\\\abs{z}\leq t}}1\leq
     2^{\omega_K(\qqq)}\sum_{\substack{z \in \aaa\\\abs{z}\leq
         4\sqrt{\N\qqq}\N\aaa}}1 \ll 2^{\omega_K(\qqq)}\sqrt{\N\qqq},
   \end{equation*}
   so the lemma holds. We assume from now on that $t \geq
   4\sqrt{\N\qqq}\N\aaa$. Define
   \begin{equation*}
     \delta := \frac{\N\qqq^{1/3}\N\aaa^{2/3}}{t^{1/6}} < \sqrt{t}/2.
   \end{equation*}
   Let $\alpha' \in \OO_K$ with $\alpha'\equiv 0 \bmod \aaa$ and $\alpha'
   \equiv\alpha\bmod \qqq$. By the Chinese remainder theorem, we have
    \begin{equation}\label{eq:circle_problem_Z}
      Z = \sum_{\substack{\rho \bmod \qqq\\\rho\OO_K + \qqq =
          \OO_K}}\sum_{v \in \aaa\qqq}\chi_{t^{1/2}D}(v+\alpha'\rho^2)\text,
    \end{equation}
    where $\chi_{rD}$ is the characteristic function of the disc
    \begin{equation*}
      rD := \{z
      \in \RR^2 \mid |z| \leq r\}.
    \end{equation*}
    Let $\psi:\RR^2\to [0,\infty)$ be a bump function for $D$, that is, $\psi
    \in \CCC^\infty(\RR^2)$, $\psi(z)=0$ for $z \not\in D$, and
    $\int_{\RR^2}\psi \dd z = 1$, and write
    $\psi_\delta(z):=\delta^{-2}\psi(\delta^{-1}z)$. Then $\psi_\delta$ is a
    bump function for $\delta D$. We define
    \begin{equation*}
      F^\pm_\delta:=\chi_{(t^{1/2}\pm \delta)D}*\psi_\delta,
    \end{equation*}
    where $*$ denotes the usual convolution of
    functions. Then $F^\pm_\delta$ are Schwartz functions and
    \begin{equation}\label{eq:circle_problem_lower_upper_bound}
      F^-_\delta(z) \leq \chi_{t^{1/2}D}(z) \leq F^+_\delta(z)\text{ for all }z
      \in \RR^2\text.
    \end{equation}
    Using \eqref{eq:circle_problem_Z},
    \eqref{eq:circle_problem_lower_upper_bound}, and the Poisson summation
    formula, we obtain
    \begin{equation}\label{eq:circle_problem_upper_lower_estimate}
      \frac{1}{\det(\aaa\qqq)}\sum_{w \in
        \dual{(\aaa\qqq)}}\ft{F^-_\delta}(w)S(w) \leq
      Z \leq  \frac{1}{\det(\aaa\qqq)}\sum_{w \in
        \dual{(\aaa\qqq)}}\ft{F^+_\delta}(w)S(w)
    \end{equation}
    where $\ft{F^\pm_\delta}$ is the Fourier transform of $F^\pm_\delta$ and
    \begin{equation*}
      S(w) := \sum_{\substack{\rho \bmod
          \qqq\\\rho\OO_K + \qqq = \OO_K}}e^{2\pi i \ip{\alpha'\rho^2}{w}}.
  \end{equation*}
  By properties of the Fourier transform, 
  \begin{equation*}
    \ft{F^\pm_\delta}(w) = (\sqrt{t}\pm \delta)^2\ft{\chi_D}((\sqrt{t}\pm\delta)w)\ft{\psi}(\delta w)\text.
  \end{equation*}
  Clearly, $\ft{\chi_D}(0) = \pi$, $\ft{\psi}(0) = 1$, and
  $S(0)=\phi_K(\qqq)$, so the summands corresponding to $w=0$ in the
  upper and lower bound from \eqref{eq:circle_problem_upper_lower_estimate} are
  \begin{equation*}
    \frac{\pi (\sqrt{t}\pm\delta)^2 \phi_K(\qqq)}{\det(\aaa\qqq)} = \frac{2 \pi
      \phi_K^*(\qqq) t}{\sqrt{|\Delta_K|}\N\aaa} +
    O\left(\frac{\sqrt{t}\delta}{\N\aaa}\right) = \frac{2 \pi
      \phi_K^*(\qqq) t}{\sqrt{|\Delta_K|}\N\aaa} + O\left(\left(\frac{\N\qqq t}{\N\aaa}\right)^{1/3}\right)\text.
  \end{equation*}
  This gives the correct main term and an acceptable error term. To prove the
  theorem, we need to bound the sums
  \begin{equation}\label{eq:circle_problem_error_sums}
   \frac{(\sqrt{t}\pm\delta)^2}{\det(\aaa\qqq)}\sum_{\substack{w \in
       \dual{(\aaa\qqq)}\\w \neq 0}}\ft{\chi_D}((\sqrt{t}\pm\delta)w)\ft{\psi}(\delta w)S(w).
  \end{equation}
  For $|w| > 0$, it is well known that
  \begin{equation*}
    \ft{\chi_D}(w) = |w|^{-1}J_1(2 \pi
    |w|) \ll |w|^{-1}\min\{1, |w|^{-1/2}\} \leq |w|^{-3/2},
  \end{equation*}
  where $J_1$ is the first-order Bessel function of the first kind. Moreover,
  $\ft{\psi}$ is a Schwartz function, so 
  \begin{equation*}
    \ft{\psi}(w) \ll
    \min\{1,|w|^{-1}\}.
  \end{equation*}
  Hence, the sums in \eqref{eq:circle_problem_error_sums} are
  \begin{align*}
    &\ll \frac{t^{1/4}}{\N(\aaa\qqq)}\sum_{\substack{w \in
      \dual{(\aaa\qqq)}\\w\neq 0}}|w|^{-3/2}\min\{1,(\delta|w|)^{-1}\}|S(w)|\\
    &\ll \frac{t^{1/4}}{\N(\aaa\qqq)}\sum_{\substack{w \in
      \dualtr{(\aaa\qqq)}\\w\neq 0}}|w|^{-3/2}\min\{1,(\delta|w|)^{-1}\}|S(2\overline{w})|\text.
  \end{align*}
   Since $w\alpha' \in \dualtr{(\aaa\qqq)}\aaa = \dualtr{\qqq}$, we can apply
   Lemma \ref{lem:exp_sum_2} to bound
   \begin{equation*}
     S(2\overline{w}) = \sum_{\substack{\rho \bmod
         \qqq\\\rho\OO_K + \qqq = \OO_K}}e^{2\pi i \Tr(w\alpha'\rho^2)} \ll \N(w\alpha'\qqq\DD_K+\qqq)^{(1-\epsilon)/2}\N\qqq^{1/2+\epsilon}.
   \end{equation*}
  This allows us to bound the sums in \eqref{eq:circle_problem_error_sums} by
  \begin{align*}
    &\ll \frac{t^{1/4}}{\N\aaa\N\qqq^{1/2-\epsilon}}\sum_{\bbb
    \mid \qqq}\N\bbb^{(1-\epsilon)/2}\sum_{\substack{w \in \dualtr{(\aaa\qqq)},
      w \neq 0\\w\alpha'\qqq\DD_K + \qqq = \bbb}}|w|^{-3/2}\min\{1,(\delta|w|)^{-1}\}.
  \end{align*}
  Now $w\alpha'\qqq\DD_K =
  w(\dualtr{(\aaa\qqq)})^{-1}\alpha'\aaa^{-1}$ and $\alpha'\aaa^{-1} +
  \qqq = \OO_K$, so the conditions under the inner sum imply $w \in
  \dualtr{(\aaa\qqq)}\bbb = \dualtr{(\aaa\qqq/\bbb)}$, and we further
  estimate
  \begin{align*}
  &\ll\frac{t^{1/4}}{\N\aaa\N\qqq^{1/2-\epsilon}}\sum_{\bbb
    \mid \qqq}\N\bbb^{(1-\epsilon)/2}\sum_{\substack{w \in \dualtr{(\aaa\qqq/\bbb)},
      w \neq 0}}|w|^{-3/2}\min\{1,(\delta|w|)^{-1}\}\text.
  \end{align*}
  For each $\bbb \mid \qqq$, let $\varphi_\bbb:\CC\to\CC$ be a map as in Lemma \ref{lem:param_ideal} with
  $\varphi_\bbb(\ZZ[i]) = \dualtr{(\aaa\qqq/\bbb)}$. Then the above expression
  is
  \begin{align*}
    &= \frac{t^{1/4}}{\N\aaa\N\qqq^{1/2-\epsilon}}\sum_{\bbb \mid
      \qqq}\N\bbb^{(1-\epsilon)/2}\sum_{\substack{v \in \ZZ[i],
        v \neq 0}}|\varphi_\bbb(v)|^{-3/2}\min\{1,(\delta|\varphi_\bbb(v)|)^{-1}\}\\
    &\ll \frac{t^{1/4}\N\qqq^{1/4+\epsilon}}{\N\aaa^{1/4}}\left(\sum_{\bbb \mid
        \qqq}\N\bbb^{-1/4-\epsilon/2}\right)\sum_{\substack{v \in \ZZ[i], v
        \neq
        0}}|v|^{-3/2}\min\left\{1,\frac{(\N\aaa\N\qqq)^{1/2}}{\delta|v|}\right\}\\
    &\ll \frac{t^{1/4}\N\qqq^{1/4+2\epsilon}}{\N\aaa^{1/4}}\left(\sum_{\substack{v
        \in \ZZ[i]\\1 \leq |v| \leq (\N\aaa\N\qqq)^{1/2}\delta^{-1}}}\!\!|v|^{-3/2}
    + \frac{(\N\aaa\N\qqq)^{1/2}}{\delta}\sum_{\substack{v \in
        \ZZ[i]\\|v| > (\N\aaa\N\qqq)^{1/2}\delta^{-1}}}\!\!|v|^{-5/2}\right)\\
    &\ll \frac{t^{1/4}\N\qqq^{1/4+2\epsilon}}{\N\aaa^{1/4}} \cdot
    \frac{(\N\aaa\N\qqq)^{1/4}}{\delta^{1/2}} = \left(\frac{t}{\N\aaa}\right)^{1/3}\N\qqq^{1/3+2\epsilon}.\qedhere
  \end{align*}
  \end{proof}

\section{An improved second summation}\label{sec:second_sum}

  In this section, we use Theorem
  \ref{thm:circle_problem} to obtain an improved error term in
  \cite[Proposition 6.1]{arXiv:1302.6151} (for $n=2$). To this end, let us briefly recall
  the setup from there: 

  We consider a nonzero fractional ideal $\OO$ of $K$, a nonzero ideal $\id q$
  of $\OO_K$, and $A \in K$ with $v_\p(A \OO) = 0$ for all prime ideals $\p$
  dividing $\id q$.

  Let $\vartheta :
  \II_K \to \RR$ be a function satisfying
  \begin{equation}\label{eq:second_sum_thetabound}
    \sum_{\substack{\aaa \in \II_K\\\N\aaa \le t}} |(\vartheta*\mu_K)(\aaa)|\cdot \N\aaa \ll c_\vartheta t(\log(t+2))^{C}
  \end{equation}
  for all $t > 0$, with constants $c_\vartheta > 0$ and $C \geq 0$. We write
  \begin{equation*}
    \mathcal{A}(\vartheta(\aaa), \aaa, \id q) := \sum_{\substack{\aaa \in \II_K\\\aaa+\id q = \OO_K}}\frac{(\vartheta*\mu_K)(\aaa)}{\N\aaa}\text.
  \end{equation*}
  For $1 \leq t_1 \leq t_2$, we consider a function $g : [t_1, t_2] \to \RR$
  for which there exists a partition of $[t_1,t_2]$ into at most $R(g)$ intervals on
  whose interior $g$ is continuously differentiable and monotonic.  Moreover,
  with constants $c_g > 0$ and $a \leq 0$, we assume that
  \begin{equation}\label{eq:second_sum_gbound}
    |g(t)| \ll c_g t^a \text{ on }[t_1, t_2].
  \end{equation}
  In \cite[Proposition 6.1]{arXiv:1302.6151}, we proved an asymptotic formula
  for the sum
  \begin{equation*}
    S(t_1, t_2) := \sum_{\substack{z \in \OO^{\neq 0}\\t_1 < \N(z\OO^{-1}) \leq t_2}}\vartheta(z\OO^{-1})\sum_{\substack{\rho \bmod \id q\\\rho\OO_K+\id q = \OO_K\\\rho^2 \equiv_{\id q} A z}}g(\N(z\OO^{-1}))\text.
  \end{equation*}
  Here, we improve the error terms to obtain the following result.
  \begin{prop}\label{prop:second_summation_improved}
    Let $\epsilon > 0$. Under the above assumptions, we have
    \begin{multline*}
      S(t_1, t_2) =
      \frac{2\pi}{\sqrt{|\Delta_K|}}\phi^*(\qqq)\mathcal{A}(\vartheta(\aaa),
      \aaa, \qqq)\int_{t_1}^{t_2}g(t)\dd t\\ + O(c_\vartheta
      c_g(\N\qqq^{1/3+\epsilon}\mathcal{E}_1 +
      2^{\omega_K(\qqq)}\N\qqq^{1/2}\mathcal{E}_2))\text,
    \end{multline*}
    where
   \begin{equation*}
     \mathcal{E}_1 \ll_{a, C}R(g)
     \begin{cases}
       \sup_{t_1 \leq t \leq t_2}(t^{a+1/3})&\text{ if } a \neq -1/3,\\
       \log(t_2+2) &\text{ if } a = -1/3,
     \end{cases}
   \end{equation*}
   and
   \begin{equation*}
     \mathcal{E}_2 \ll_{a, C}R(g)
     \begin{cases}
       t_1^a\log(t_1+2)^{C+1} &\text{ if } a \neq 0,\\
       \log(t_2+2)^{C+1} &\text{ if } a = 0.
     \end{cases}
   \end{equation*}
   Moreover, the same formula holds if, in the definition of $S(t_1, t_2)$, the
   range $t_1 < \N(z\OO^{-1}) \leq t_2$ is replaced by $t_1 \leq \N(z\OO^{-1})
   \leq t_2$.
 \end{prop}

 The proof is analogous to the proof of \cite[Proposition
 6.1]{arXiv:1302.6151}, except that we use the lemma below instead of
 \cite[Lemma 6.4]{arXiv:1302.6151}.

 \begin{lemma}\label{lem:sum_arith_function}
    Let $\aaa$, $\qqq$ be ideals of $\OO_K$ and let $\alpha \in \OO_K$
    with $\aaa+\qqq = \alpha\OO_K+\qqq = \OO_K$, and
    $\epsilon>0$. Then, for $t \geq 0$,
    \begin{multline*}
      \sum_{\substack{\rho \bmod \qqq\\\rho\OO_K + \qqq =
          \OO_K}}\sum_{\substack{z \in \aaa^{\neq 0}\\z \equiv
          \alpha\rho^2 \bmod \qqq\\\abs{z}\leq
          t\N\aaa}}\vartheta(z\aaa^{-1}) = \frac{2 \pi}{\sqrt{|\Delta_K|}}
        \phi_K^*(\qqq)\mathcal{A}(\vartheta(\bbb),\bbb,\qqq)t\\
      +
      O_{C}\left(c_\vartheta\left(t^{1/3}\N\qqq^{1/3+\epsilon}
          +
          2^{\omega_K(\qqq)}\N\qqq^{1/2}\log(t+2)^{C+1}\right)\right).
    \end{multline*}
  \end{lemma}

  \begin{proof}
    We proceed similarly to the proof of \cite[Lemma 6.4]{arXiv:1302.6151}. Let
    $L$ be the expression on the left-hand side. Then
    \begin{align*}
      L = \sum_{\N\bbb \leq
        t}(\vartheta*\mu_K)(\bbb)\sum_{\substack{\rho \bmod
          \qqq\\\rho\OO_K + \qqq = \OO_K}}\sum_{\substack{z \in
          \aaa\bbb^{\neq 0}\\z \equiv \alpha\rho^2 \bmod
          \qqq\\\abs{z}\leq t\N\aaa}}1.
    \end{align*}
   The inner sum is zero whenever $\bbb+\qqq\neq\OO_K$ and can be estimated by
   Theorem \ref{thm:circle_problem} otherwise. Hence,
   \begin{align*}
     L = \sum_{\substack{\N\bbb \leq
        t\\\bbb+\qqq=\OO_K}}(\vartheta*\mu_K)(\bbb)\left(\frac{2\pi
        \phi_K^*(\qqq) t}{\sqrt{|\Delta_K|}\N\bbb} +
      O\left(\left(\frac{t}{\N\bbb}\right)^{1/3}\N\qqq^{1/3+\epsilon} + 2^{\omega_K(\qqq)}\N\qqq^{1/2} \right)\right).
   \end{align*}
   This gives the main term in the lemma plus an error term
   $\ll$
   \begin{equation*}
     t\sum_{\N\bbb>t}\!\frac{|(\vartheta*\mu_K)(\bbb)|}{\N\bbb} +
     t^{1/3}\N\qqq^{1/3+\epsilon}\sum_{\N\bbb\leq t}\!\frac{|(\vartheta*\mu_K)(\bbb)|}{\N\bbb^{1/3}}
     + 2^{\omega_K(\qqq)}\N\qqq^{1/2}\sum_{\N\bbb\leq t}\!|(\vartheta*\mu_K)(\bbb)|\text.
   \end{equation*}
   The first part is $\ll_C c_\vartheta\log(t+2)^C$, the second part is $\ll_C
   c_\vartheta t^{1/3}\N\qqq^{1/3+\epsilon}$, and the third part is $\ll_C
   c_\vartheta 2^{\omega_K(\qqq)}\N\qqq^{1/2}\log(t+2)^{C+1}$.
  \end{proof}

\begin{proof}[Proof of Proposition \ref{prop:second_summation_improved}]
  With
  \begin{equation*}
    \tilde{\vartheta}(\aaa) := \vartheta(\aaa)\sum_{\substack{z \in \OO^{\neq
          0}\\z\OO^{-1}=\aaa }}\sum_{\substack{\rho \bmod \qqq\\\rho\OO_K +
        \qqq = \OO_K\\\rho^2 \equiv_\qqq Az}}1\text,
  \end{equation*}
  we have
  \begin{equation*}
    S(t_1, t_2) = \sum_{\substack{\aaa \in [\OO^{-1}]\cap
        \II_K\\t_1<\N\aaa\leq t_2}}\tilde{\vartheta}(\aaa)g(\N\aaa)\text.
  \end{equation*}
  Let $A_1 \in \OO^{-1}$, $A_2 \in \OO_K$, such that $A = A_1/A_2$ and $A_1\OO
  + \qqq = A_2\OO_K + \qqq = \OO_K$. Then, for $t \geq 0$,
  \begin{align*}
    \sum_{\substack{\aaa \in [\OO^{-1}]\cap \II_K\\\N\aaa\leq
        t}}\tilde{\vartheta}(\aaa) &= \sum_{\substack{\rho\bmod\qqq\\\rho\OO_K +
        \qqq = \OO_K}}\sum_{\substack{z \in \OO^{\neq 0}\\A_1 z \equiv
        A_2\rho^2\bmod \qqq\\\N(z\OO^{-1})\leq t}}\vartheta(z\OO^{-1})\\ &=
    \sum_{\substack{\rho\bmod\qqq\\\rho\OO_K + \qqq = \OO_K}}\sum_{\substack{A_1
        z \in A_1\OO^{\neq 0}\\A_1 z \equiv A_2\rho^2\bmod
        \qqq\\\N(A_1z(A_1\OO)^{-1})\leq t}}\vartheta(A_1z(A_1\OO)^{-1})\text.
  \end{align*}
  By Lemma \ref{lem:sum_arith_function}, the last expression is
  \begin{equation*}
    \frac{2 \pi
      \phi^*(\qqq)\mathcal{A}(\vartheta(\bbb),\bbb,\qqq)}{\sqrt{|\Delta_K|}}t
    + O_{C}\left(c_{\vartheta}\left(t^{1/3}\N\qqq^{1/3+\epsilon} + 2^{\omega_K(\qqq)}\N\qqq^{1/2}\log(t+2)^{C+1} \right)\right),
  \end{equation*}
    so the proposition follows from \cite[Lemma 2.10]{arXiv:1302.6151}.
  \end{proof}

\section{Passage to a universal torsor}\label{sec:passage}

Our parameterization of $K$-rational points on the cubic surface $S$
defined by (\ref{eq:def_E6}) derived via
\cite[Section~4]{arXiv:1302.6151} from the description of the Cox ring
of its minimal desingularization $\tS$ \cite{MR2029868, math.AG/0604194}.

\begin{figure}[ht]
  \centering
  \[\xymatrix@R=0.05in @C=0.05in{E_{10} \ar@{-}[r] \ar@2{-}[dr] \ar@3{-}[dd] & \li{7} \ar@{-}[r] & \ex{4} \ar@{-}[r] & \ex{5} \ar@{-}[dr]\\
      & E_8 \ar@{-}[r] & \ex{1} \ar@{-}[r] & \ex{3} \ar@{-}[r] & \ex{6}\\
      E_9 \ar@{-}[rrr] \ar@{-}[ur] & & & \ex{2} \ar@{-}[ur]}\]
  \caption{Configuration of curves on $\tS$.}
  \label{fig:E6_dynkin}
\end{figure}
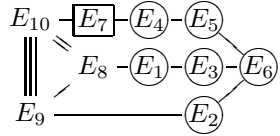

For any given $\classtuple = (C_0, \dots, C_6) \in \classrep^7$, we
define $u_\classtuple := \N(C_0^3 C_1^{-1}\cdots C_6^{-1})$ and
\begin{equation}\label{eq:E6_def_Oj}
  \begin{aligned}
    \OO_1 &:= C_1C_2^{-1},  & \OO_2 &:= C_0C_1^{-1}C_2^{-1}C_3^{-1},  & \OO_3 &:= C_2C_3^{-1}, \\
    \OO_4 &:= C_5C_6^{-1},  & \OO_5 &:= C_4C_5^{-1},  & \OO_6 &:= C_3C_4^{-1},  \\
    \OO_7 &:= C_6,  & \OO_8 &:= C_0C_1^{-1},  & \OO_9 &:= C_0, \\
    \OO_{10} &:= C_0^3C_1^{-1}C_2^{-1}C_3^{-1}C_4^{-1}C_5^{-1}C_6^{-1}.& & & &
  \end{aligned}
\end{equation}
Let
\begin{equation*}
  \OO_{j*} :=
  \begin{cases}
    \OO_j^{\neq 0}, & j \in \{1,\ldots, 7\},\\
    \OO_j, & j \in \{8, 9, 10\}.
  \end{cases}
\end{equation*}
For $\eta_j \in \OO_j$, we define
\begin{equation*}
  \eI_j := \e_j \OO_j^{-1}\text.
\end{equation*}
For $B \geq 0$, let $\mathcal{R}(B)$ be the set of all $(\e_1, \ldots, \e_9)
\in \CC^9$ with $\e_4\e_5\e_7\neq 0$ and
\begin{align}
      \abs{\e_1^2\e_2^3\e_3^4\e_4^4\e_5^5\e_6^6\e_7^3} &\leq B,\label{eq:E6_height_1}\\
      \abs{\e_1^2\e_2^2\e_3^3\e_4^2\e_5^3\e_6^4\e_7\e_8} &\leq B,\label{eq:E6_height_2}\\
      \abs{\e_1\e_2^2\e_3^2\e_4\e_5^2\e_6^3\e_9} &\leq B,\label{eq:E6_height_3}\\
      \abs{\frac{\e_1^2\e_3\e_8^3 + \e_2\e_9^2}{\e_4^2\e_5\e_7^3}} &\leq B\label{eq:E6_height_4}\text.
\end{align}
Moreover, let $M_\classtuple(B)$ be the set of all 
\begin{equation*}
  (\e_1, \ldots, \e_{10}) \in \OO_{1*} \times \cdots \times \OO_{10*} 
\end{equation*}
that satisfy the \emph{height conditions}
\begin{equation}\label{eq:E6_height}
     (\e_1, \ldots, \e_9) \in \mathcal{R}(u_\classtuple B)\text,
\end{equation}
the \emph{torsor equation}
\begin{equation}\label{eq:E6_torsor}
  \e_1^2\e_3\e_8^3 + \e_2\e_9^2 + \e_4^2\e_5\e_7^3\e_{10} = 0,
\end{equation}
and the \emph{coprimality conditions}
\begin{equation}\label{eq:E6_coprimality}
  \eI_j + \eI_k = \OO_K \text{ for all distinct nonadjacent vertices $E_j$, $E_k$ in Figure~\ref{fig:E6_dynkin}.}
\end{equation}

\begin{lemma}\label{lem:E6_passage_to_torsor}
  Let $K$ be an imaginary quadratic field. Then
  \begin{equation*}
    N_{U,H}(B) = \frac{1}{\numunits^7}\sum_{\classtuple \in \classrep^7}|M_\classtuple(B)|\text.
  \end{equation*}
\end{lemma}

\begin{proof}
  The lemma is a special case of \cite[Claim 4.1]{arXiv:1302.6151},
  which we prove by proving first \cite[Claim 4.2]{arXiv:1302.6151},
  starting from the curves $E_9^{(0)}= \{y_0=0\}$, $E_2^{(0)} :=
  \{y_1=0\}$, $E_8^{(0)} := \{y_2 = 0\}$, $E_{10}^{(0)} := \{-y_0^2y_1 -
  y_2^3 = 0\}$ in $\PP_K^2$, for the sequence of blow-ups
  \begin{enumerate}
  \item blow up $E_2^{(0)}\cap E_8^{(0)} \cap E_{10}^{(0)}$, giving $E_1^{(1)}$,
  \item blow up $E_1^{(1)}\cap E_2^{(1)} \cap E_{10}^{(1)}$, giving $E_3^{(2)}$,
  \item blow up $E_2^{(2)}\cap E_3^{(2)} \cap E_{10}^{(2)}$, giving $E_6^{(3)}$,
  \item blow up $E_6^{(3)} \cap E_{10}^{(3)}$, giving $E_5^{(4)}$,
  \item blow up $E_5^{(4)} \cap E_{10}^{(4)}$, giving $E_4^{(5)}$,
  \item blow up $E_4^{(5)} \cap E_{10}^{(5)}$, giving $E_7^{(6)}$,
  \end{enumerate}

  With the inverse $\pi\circ\rho^{-1} : \PP^2_K \dasharrow S$ of the
  projection $\rho\circ\pi^{-1} : S \rto \PP^2_K$, $(x_0 : \cdots : x_4)
  \mapsto (x_0 : x_2 : x_3)$ given by
  \begin{equation*}
    (y_0 : y_1 : y_2) \mapsto (y_0y_1^2 : -y_0^2y_1-y_2^3 : y_1^3 : y_1^2y_2)\text,
  \end{equation*}
  and the map $\Psi$ from \cite[Claim 4.2]{arXiv:1302.6151} sending $(\e_1,
  \ldots, \e_{10})$ to
  \begin{equation*}
    (\e_1\e_2^2\e_3^2\e_4\e_5^2\e_6^3\e_9, \e_{10}, \e_1^2\e_2^3\e_3^4\e_4^4\e_5^5\e_6^6\e_7^3, \e_1^2\e_2^2\e_3^3\e_4^2\e_5^3\e_6^4\e_7\e_8),
  \end{equation*}
  we see that the requirements of \cite[Lemma 4.3]{arXiv:1302.6151} are
  satisfied, so \cite[Claim 4.2]{arXiv:1302.6151} holds for $i=0$.

  We apply \cite[Remark~4.5]{arXiv:1302.6151} for steps (1), (2), (3).
  For (1), we define $\e_1'' \in C_1$ with
  $[I_2'+I_8'+I_{10}']=[C_1^{-1}]$ such that $I_1''=I_2'+I_8'+I_{10}'$. We
  use the relation $\e_1''^2\e_8''^3+\e_2''\e_9''^2+\e_{10}''=0$ to
  check the coprimality conditions for $\e_1'', \e_2'', \e_8'',
  \e_{10}''$, namely $\eI_2''+\eI_8''=\OO_K$ (this holds because of the
  relation and $\eI_2''+\eI_8''+\eI_{10}''=\OO_K$ by construction) and
  $\eI_1''+\eI_8''+\eI_{10}''=\OO_K$ (this holds because of the relation
  and $\eI_2''+\eI_8''+\eI_{10}''=\OO_K$ by construction and the
  coprimality condition $I_1''+I_9''=\OO_K$ provided by the proof of
  \cite[Lemma~4.4]{arXiv:1302.6151}).

  For (2), we define $\e_3'' \in C_2$ with $[I_1'+I_2'+I_{10}']=[C_2^{-1}]$ such
  that $I_3''=I_1'+I_2'+I_{10}'$. The relation is
  $\e_1''^2\e_3''\e_8''^3+\e_2''\e_9''^2+\e_{10}''=0$. We check the coprimality
  conditions $I_1''+I_2''=\OO_K$ (this holds because of the relation and
  $I_1''+I_2''+I_{10}''=\OO_K$ by construction) and $I_1''+I_{10}''=\OO_K$ (this
  holds because of the relation and $I_1''+I_2''=\OO_K$ as just shown and
  $I_1''+I_9''=\OO_K$ as before).

  For (3), we define $\e_6'' \in C_3$ with
  $[I_2'+I_3'+I_{10}']=[C_3^{-1}]$ such that
  $I_6''=I_2'+I_3'+I_{10}'$. The relation is
  $\e_1''^2\e_3''\e_8''^3+\e_2''\e_9''^2+\e_{10}''=0$. We check the
  coprimality conditions $I_2''+I_3''=\OO_K$ (this holds because of
  the relation and $I_2''+I_3''+I_{10}''=\OO_K$ by construction),
  $I_2''+I_{10}''=\OO_K$ (this holds because of the relation and
  $I_2''+I_3''=\OO_K$ as just shown $I_1''+I_2''=\OO_K$ as before and
  $I_2''+I_8''=\OO_K$ as before) and $I_3''+I_{10}''=\OO_K$ (this
  holds because of the relation and $I_2''+I_{10}''=\OO_K$ as just
  shown and $I_3''+I_9''=\OO_K$ by the proof of
  \cite[Lemma~4.4]{arXiv:1302.6151}).

  For (4), (5), (6), we can apply \cite[Lemma~4.4]{arXiv:1302.6151}. This proves
  \cite[Claim 6.2]{arXiv:1302.6151}, and we deduce \cite[Claim
  6.1]{arXiv:1302.6151} as in \cite[Lemma 9.1]{arXiv:1302.6151}.
\end{proof}

\section{Summations}\label{sec:summations}
\subsection{The first summation over $\e_9$ with dependent $\e_{10}$}
Let $\ee' := (\e_1, \ldots, \e_8)$ and $\eII' := (\eI_1, \ldots, \eI_8)$. Let
$\theta_0(\eII') := \prod_\p\theta_{0,\p}(J_\p(\eII'))$, with $J_\p(\eII') :=
\{j \in \{1, \ldots, 8\}\ :\  \p \mid \eI_j\}$ and
\begin{equation*}
  \theta_{0,\p}(J):=
  \begin{cases}
    1 &\text{ if } J = \emptyset, \{1\}, \{2\}, \{3\}, \{4\}, \{5\}, \{6\},
    \{7\}, \{8\}\\
    \ &\text{ or } J= \{1,3\}, \{1,8\}, \{2,6\}, \{3,6\}, \{4,5\}, \{4,7\}, \{5,6\}\\
    0 &\text{ otherwise.} 
  \end{cases}
\end{equation*}
Then $\theta_{0}(\eII') = 1$ if and only if $\eI_1$, $\ldots$, $\eI_8$
satisfy the coprimality conditions from \eqref{eq:E6_coprimality}, and
$\theta_{0}(\eII') = 0$ otherwise.

\begin{lemma}\label{lem:E6_first_summation}
  We have
  \begin{equation*}
    |M_\classtuple(B)| = \frac{2}{\sqrt{|\Delta_K|}}\sum_{\ee' \in \OO_{1*}
      \times \dots \times \OO_{8*}} \theta_9(\ee',
    \classtuple)V_9(\N\eI_1,\ldots,\N\eI_8; B) + O_\classtuple(B(\log B)^2),
  \end{equation*}
  where
  \begin{equation*}
    V_9(t_1, \ldots, t_8; B) :=  \frac{1}{t_4^2t_5t_7^3}\int_{(\sqrt{t_1}, \ldots,
      \sqrt{t_8},\e_9) \in \mathcal{R}(B)} \dd \e_9.
  \end{equation*}
  Moreover,
  \begin{equation*}
    \theta_9(\ee', \classtuple) := \sum_{\substack{\kc \mid \eI_4\eI_5\eI_6\\\kc + 
        \eI_2\eI_3 =
        \OO_K}}\frac{\mu_K(\kc)}{\N\kc}\tilde\theta_9(\eII',\kc)
    \sum_{\substack{\rho \bmod \kc\eI_4^2\eI_5\eI_7^3\\\rho\OO_K + \kc\eI_4^2\eI_5\eI_7^3 = \OO_K\\\rho^2 \equiv_{\kc\eI_4^2\eI_5\eI_7^3}         \e_8 A}}1\text,
  \end{equation*}
  with
  \begin{equation*}
  \tilde\theta_9(\eII',\kc) := \theta_0(\eII')\frac{\phi_K^*(\eI_1\eI_3\eI_6)}{\phi_K^*(\eI_6+
      \kc\eI_5)}\text.
  \end{equation*}
  Here, $A = A(\e_2, \ldots, \e_7) := -\e_3/(t^2\e_2)$, for a $t =
  t(\e_4,\e_5,\e_6,\e_7) \in K^\times$ such that
  $t\OO_1\OO_8\OO_9^{-1}$ is a prime ideal not dividing
  $\eI_4\eI_5\eI_6\eI_7$. Moreover, $\e_8 A$ is invertible
  modulo $\kc\eI_4^2\eI_5\eI_7^3$ whenever $\theta_0(\eII')\neq 0$.
\end{lemma}

\begin{proof}
  For fixed $\e_1, \dots, \e_8$, the first summation estimates the
  number of $\e_9,\e_{10}$ with $(\e_1, \dots, \e_{10}) \in
  M_\classtuple(B)$. This is considered in general in
  \cite[Proposition 5.3]{arXiv:1302.6151}, which we apply with the
  following data: $(A_1, A_2, A_0) = (3,1,8)$, $(B_1, B_0)=(2,9)$,
  $(C_1, C_2, C_3, C_0) = (5,4,7,10)$, $D = 6$, and $u_\classtuple B$
  instead of $B$. Moreover, we define $\At =
  \At(\e_4,\e_5,\e_6,\e_7):=\e_1\e_8t$ and $\Ao =
  \Ao(\e_4,\e_5,\e_6,\e_7):=\e_3\e_8t^{-2}$, so $\e_1^2\e_3\e_8^3 =
  \Ao\At^2$. We obtain a main term, which is the one given in the
  statement of this lemma, and an error term, which we still need to
  sum over $\e_1, \dots, \e_8$.

  Let us consider the error term. For given $\ee'$, the set of all
  $\e_9$ with $(\e_1, \ldots, \e_9) \in \mathcal{R}(u_\classtuple B)$
  is contained in two balls of radius
  \begin{equation*}
    R(\ee'; u_\classtuple B) \ll_\classtuple
    \begin{cases}
      (B^3\N(\eI_4^2\eI_5\eI_7^3)^3/\N(\eI_1^2\eI_2^2\eI_3\eI_8^3))^{1/8}&\text{ if }\e_8\neq 0\\
      (B/\N(\eI_1\eI_2^2\eI_3^2\eI_4\eI_5^2\eI_6^3))^{1/2} &\text{ if }\e_8 =
      0\text.
    \end{cases}
  \end{equation*}
  If $\eta_8 \neq 0$ this follows from taking the geometric mean of
  both expressions in the minimum in \cite[Lemma 3.5,
  (1)]{arXiv:1302.6151} applied to \eqref{eq:E6_height_4}. If $\eta_8
  = 0$ then it follows from \eqref{eq:E6_height_3}.  Thus, the error
  term is
  \begin{equation}\label{eq:first_error}
    \ll \sum_{\ee'\text{, }\eqref{eq:E6_first_height_cond_ideals}\text{, }\eqref{eq:E6_second_height_cond_ideals}}2^{\omega_K(\eI_1\eI_3\eI_6)+\omega_K(\eI_4\eI_5\eI_6)+\omega_K(\eI_4\eI_5\eI_6\eI_7)}\left(\frac{R(\ee'; u_\classtuple B)}{\N(\eI_4^2\eI_5\eI_7^3)^{1/2}}+1\right)\text,
  \end{equation}
  where, using \eqref{eq:E6_height_1} and \eqref{eq:E6_height_2}, the
  sum runs over all $\ee' \in \OO_{1*}\times\cdots\times\OO_{7*}$ with
  \begin{align}\label{eq:E6_first_height_cond_ideals}
    \N(\eI_1^2\eI_2^3\eI_3^4\eI_4^4\eI_5^5\eI_6^6\eI_7^3) &\leq B \text{, and }\\
    \N(\eI_1^2\eI_2^2\eI_3^3\eI_4^2\eI_5^3\eI_6^4\eI_7\eI_8) &\leq B
    \text.\label{eq:E6_second_height_cond_ideals}
  \end{align}

  The sum of the first term of (\ref{eq:first_error}) over all $\ee'$ with $\e_8 \neq 0$ is bounded by
  \begin{align*}
    &\ll_\classtuple \sum_{\eII'\text{, }\eqref{eq:E6_second_height_cond_ideals}}\frac{2^{\omega_K(\eI_1\eI_3\eI_6)+\omega_K(\eI_4\eI_5\eI_6)+\omega_K(\eI_4\eI_5\eI_6\eI_7)}B^{3/8}}{(\N\eI_1^2\N\eI_2^2\N\eI_3\N\eI_4^2\N\eI_5\N\eI_7^3\N\eI_8^3)^{1/8}}\\
    &\ll\sum_{\substack{\eI_1, \dots, \eI_7\\\N\eI_j\leq B}}\frac{2^{\omega_K(\eI_1\eI_3\eI_6)+\omega_K(\eI_4\eI_5\eI_6)+\omega_K(\eI_4\eI_5\eI_6\eI_7)}B}{\N\eI_1^{3/2}\N\eI_2^{3/2}\N\eI_3^2\N\eI_4^{3/2}\N\eI_5^2\N\eI_6^{5/2}\N\eI_7}\\
    &\ll B(\log B)^2,
  \end{align*}
  and the sum of the second term of (\ref{eq:first_error}) over all $\ee'$ with $\e_8 \neq 0$ is bounded by
  \begin{align*}
    &\ll_\classtuple \sum_{\eII'\text{, }\eqref{eq:E6_second_height_cond_ideals}}2^{\omega_K(\eI_1\eI_3\eI_6)+\omega_K(\eI_4\eI_5\eI_6)+\omega_K(\eI_4\eI_5\eI_6\eI_7)}\\
    &\ll\sum_{\substack{\eI_1, \dots, \eI_7\\\N\eI_j\leq B}}\frac{2^{\omega_K(\eI_1\eI_3\eI_6)+\omega_K(\eI_4\eI_5\eI_6)+\omega_K(\eI_4\eI_5\eI_6\eI_7)}B}{\N\eI_1^2\N\eI_2^2\N\eI_3^3\N\eI_4^2\N\eI_5^3\N\eI_6^4\N\eI_7}\\
    &\ll B(\log B)^2.
  \end{align*}

  The sum of the first term of (\ref{eq:first_error}) over all $\ee'$ with $\e_8 = 0$ is bounded by
  \begin{align*}
    &\ll_\classtuple \sum_{\substack{\eI_1, \ldots,
        \eI_7\\\eqref{eq:E6_first_height_cond_ideals}}}\frac{2^{\omega_K(\eI_1\eI_3\eI_6)+\omega_K(\eI_4\eI_5\eI_6)+\omega_K(\eI_4\eI_5\eI_6\eI_7)}B^{1/2}}{(\N\eI_1\N\eI_2^2\N\eI_3^2\N\eI_4^3\N\eI_5^3\N\eI_6^3\N\eI_7^3)^{1/2}}\\
    &\ll\sum_{\substack{\eI_2, \dots, \eI_7\\\N\eI_j \leq
        1}}\frac{2^{\omega_K(\eI_3\eI_6)+\omega_K(\eI_4\eI_5\eI_6)+\omega_K(\eI_4\eI_5\eI_6\eI_7)}B^{3/4}\log
        B}{\N\eI_2^{7/4}\N\eI_3^2\N\eI_4^{5/2}\N\eI_5^{11/4}\N\eI_6^3\N\eI_7^{9/4}}
        \\
    &\ll B^{3/4}\log B,
  \end{align*}
  and the sum of the second term of (\ref{eq:first_error}) over all $\ee'$ with $\e_8 = 0$ is bounded by
  \begin{align*}
    &\ll_\classtuple \sum_{\substack{\eI_1, \ldots,
        \eI_7\\\eqref{eq:E6_first_height_cond_ideals}}}2^{\omega_K(\eI_1\eI_3\eI_6)+\omega_K(\eI_4\eI_5\eI_6)+\omega_K(\eI_4\eI_5\eI_6\eI_7)}\\
    &\ll\sum_{\substack{\eI_2, \dots, \eI_7\\\N\eI_j \leq
        1}}
        \frac{2^{\omega_K(\eI_3\eI_6)+\omega_K(\eI_4\eI_5\eI_6)+\omega_K(\eI_4\eI_5\eI_6\eI_7)}B^{1/2}\log
          B}{\N\eI_2^{3/2}\N\eI_3^2\N\eI_4^2\N\eI_5^{5/2}\N\eI_6^3\N\eI_7^{3/2}}\\
    &\ll B^{1/2}\log B.\qedhere
  \end{align*}
\end{proof}

\subsection{The second summation over $\e_8$.}

\begin{lemma}\label{lem:E6_second_summation}
  Write $\ee'' := (\e_1, \dots, \e_7)$ and $\OO'':=\OO_{1*}\times \cdots
  \times \OO_{7*}$. We have
  \begin{align*}
    M_\classtuple(B) &=\left(\frac{2}{\sqrt{|\Delta_K|}}\right)^2
    \sum_{\ee''\in \OO''}
    \mathcal{A}(\theta_9'(\eII'),\eI_8)V_{98}(\N\eI_1,\ldots,\N\eI_7; B)\\ &+
    O_{\classtuple}(B(\log B)^2),
  \end{align*}
  where, for $t_1, \ldots, t_7 \geq 1$,
  \begin{equation*}
    V_{98}(t_1, \ldots, t_7; B) := \frac{\pi}{t_4^2t_5t_7^3}\int\limits_{\substack{(\sqrt{t_1},
        \ldots, \sqrt{t_8}, \e_9) \in \mathcal{R}(B)}} \dd t_8 \dd \e_9\text,
  \end{equation*}
  with a real variable $t_8$ and a complex variable $\e_9$.
\end{lemma}

\begin{proof}
  We follow the strategy described in \cite[Section 6]{arXiv:1302.6151} in the
  case $b_0 \geq 2$, except that we use Proposition
  \ref{prop:second_summation_improved} instead of \cite[Proposition 6.1]{arXiv:1302.6151}.  We write
  \begin{equation*}
    M_\classtuple(B) = \frac{2}{\sqrt{|\Delta_K|}}\sum_{\ee'' \in \OO''}
    \sum_{\substack{\kc \mid \eI_4\eI_5\eI_6\\\kc+\eI_2\eI_3 =
        \OO_K}}\frac{\mu(\kc)}{\N\kc}\Sigma + O_\classtuple(B(\log B)^2),
  \end{equation*}
  where
  \begin{equation*}
    \Sigma := \sum_{\substack{\e_8 \in \OO_{8*}}}\vartheta(\eI_8)\sum_{\substack{\rho \bmod        \kc\eI_4^2\eI_5\eI_7^3\\\rho\OO_K + \kc\eI_4^2\eI_5\eI_7^3 = \OO_K \\\rho^2 \equiv_{\kc\eI_4^2\eI_5\eI_7^3} \e_8 A}}g(\N\eI_8)\text,
  \end{equation*}
  with $\vartheta(\eI_8) := \tilde\theta_9(\eII', \kc)$ and $g(t) :=
  V_9(\N\eI_1,\ldots,\N\eI_7, t; B)$.
 
  By \cite[Lemma 5.5, Lemma 2.2]{arXiv:1302.6151}, the function $\vartheta$
  satisfies (\ref{eq:second_sum_thetabound}) with $C:=0$, $c_\vartheta :=
  2^{\omega_K(\eI_2\eI_3\eI_4\eI_5\eI_6\eI_7)}$.  By \eqref{eq:E6_height_2}, we
  have $g(t)=0$ if $t > t_2 :=
  B/\N(\eI_1^2\eI_2^2\eI_3^3\eI_4^2\eI_5^{3}\eI_6^4\eI_7)$, and, using Lemma
  \cite[Lemma 3.5, (2)]{arXiv:1302.6151} applied to \eqref{eq:E6_height_4}, we
  have $g(t) \ll B^{1/2}/(\N\eI_2^{1/2}\N\eI_4\N\eI_5^{1/2}
  \N\eI_7^{3/2})$. By Proposition \ref{prop:second_summation_improved}, we
  obtain
  \begin{align*}
    \Sigma &= \frac{2
      \pi}{\sqrt{|\Delta_K|}}\phi_K^*(\kc\eI_4^2\eI_5\eI_7^3)\mathcal{A}(\vartheta(\aaa),
    \aaa, \kc\eI_4^2\eI_5\eI_7^3)\int_{t \geq 1}g(t)\dd t + O(\vartheta(0)g(0))\\
    &+
    O\left(\frac{2^{\omega_K(\eI_2\eI_3\eI_4\eI_5\eI_6\eI_7)}B^{1/2}}{\N(\eI_2\eI_4^2\eI_5\eI_7^3)^{1/2}}\left(\frac{B^{1/3}\N(\kc\eI_4^2\eI_5\eI_7^3)^{1/3+\epsilon}}{\N(\eI_1^{2}\eI_2^{2}\eI_3^{3}\eI_4^{2}\eI_5^3\eI_6^{4}\eI_7)^{1/3}}
        + \frac{2^{\omega_K(\kc\eI_4\eI_5\eI_7)}\log
        B}{\N(\kc\eI_4^2\eI_5\eI_7^3)^{-1/2}}\right)\right)\text.
  \end{align*}
  Clearly, $\int_0^1g(t)\dd t$ and the error term $O(\vartheta(0)g(0))$ are
  dominated by the other error term.  Using \cite[Lemma 6.3]{arXiv:1302.6151}
  we see that the main term in the lemma is correct.

  For the error term, we may sum over $\kc$ and over the ideals $\eI_j$ instead
  of the $\e_j$, since $|\OO_K^\times| < \infty$. By \eqref{eq:E6_height_1}, it
  suffices to sum over $\kc$ and all $(\eI_1, \ldots, \eI_7)$ satisfying
  \begin{equation}\label{eq:E6_height_1_ideals}
    \N\eI_1^2\N\eI_2^3\N\eI_3^4\N\eI_4^4\N\eI_5^5\N\eI_6^6\N\eI_7^3 \leq B\text.
  \end{equation}
  Thus, the total error is bounded by the sum of
  \begin{align*}
    &\sum_{\substack{\eI_1, \ldots,
        \eI_7\\\eqref{eq:E6_height_1_ideals}}}\frac{2^{2\omega_K(\eI_2\eI_3\eI_4\eI_5\eI_6\eI_7)}B^{5/6}}{\N\eI_1^{2/3}\N\eI_2^{7/6}\N\eI_3\N\eI_4^{1-2\epsilon}\N\eI_5^{7/6-\epsilon}\N\eI_6^{4/3}\N\eI_7^{5/6-3\epsilon}}
      \\
    &\ll \sum_{\substack{\eI_2, \ldots, \eI_7\\\N\eI_j \leq B}}\frac{2^{2\omega_K(\eI_2\eI_3\eI_4\eI_5\eI_6\eI_7)}B}{\N\eI_2^{5/3}\N\eI_3^{5/3}\N\eI_4^{5/3-2\epsilon}\N\eI_5^{2-\epsilon}\N\eI_6^{7/3}\N\eI_7^{4/3-3\epsilon}}\\
    &\ll B
  \end{align*}
  (whenever $\epsilon < 1/9$; we choose $\epsilon := 1/18$) and
  \begin{align*}
    &\sum_{\substack{\eI_1, \ldots,
        \eI_7\\\eqref{eq:E6_height_1_ideals}}} \frac{2^{3\omega_K(\eI_2\eI_3\eI_4\eI_5\eI_6\eI_7)}B^{1/2}\log B}{\N\eI_2^{1/2}}\\
    &\ll \sum_{\substack{\eI_2, \ldots, \eI_7\\\N\eI_j \leq B}}\frac{2^{3\omega_K(\eI_2\eI_3\eI_4\eI_5\eI_6\eI_7)}B\log B}{\N\eI_2^{2}\N\eI_3^2\N\eI_4^{2}\N\eI_5^{5/2}\N\eI_6^3\N\eI_7^{3/2}}\\
    &\ll B\log B.\qedhere
  \end{align*}
\end{proof}

\begin{lemma}\label{lem:E6_second_summation_ideals}
  If $\eII''$ runs over all seven-tuples $(\eI_1, \ldots, \eI_7)$ of
  nonzero ideals of $\OO_K$ then we have
  \begin{multline*}
    N_{U,H}(B) =
    \left(\frac{2}{\sqrt{|\Delta_K|}}\right)^2\sum_{\eII''}\mathcal{A}(\theta_9'(\eII'),\eI_8)V_{98}(\N\eI_1,
    \ldots, \N\eI_7; B) + O(B(\log B)^5)\text.
  \end{multline*}
\end{lemma}

\begin{proof}
  This is analogous to \cite[Lemma 9.4]{arXiv:1302.6151}.
\end{proof}

\subsection{The remaining summations}

\begin{lemma}\label{lem:E6_completion}
  We have
  \begin{multline*}
    N_{U,H}(B) = \left(\frac{2}{\sqrt{|\Delta_K|}}\right)^9
    \left(\frac{h_K}{\numunits}\right)^7 \prod_\p
    \left(1-\frac{1}{\N\p}\right)^7\left(1+\frac{7}{\N\p}+\frac{1}{\N\p^2}\right)
    V_0(B)\\ + O(B(\log B)^5\log \log B),
  \end{multline*}
  where
  \begin{equation*}
    V_0(B) := \int\limits_{\substack{(\e_1, \ldots, \e_9)\in\mathcal{R}(B)\\\abs{\e_1}, \ldots, \abs{\e_7} \ge 1}}\frac{1}{\abs{\e_4^2\e_5\e_7^3}}\dd \e_1 \cdots \dd \e_9,
  \end{equation*}
 with complex variables $\e_1, \ldots, \e_9$.
\end{lemma}

\begin{proof}
  By \cite[Lemma 3.5, (5)]{arXiv:1302.6151} applied to \eqref{eq:E6_height_4}, we have
  \begin{equation*}
    V_{98}(t_1, \ldots, t_7; B) \ll \frac{B^{5/6}}{t_1^{2/3}t_2^{1/2}t_3^{1/3}t_4^{1/3}t_5^{1/6}t_7^{1/2}}
    =\frac{B}{t_1 \cdots
      t_7}\left(\frac{B}{t_1^2t_2^3t_3^4t_4^4t_5^5t_6^6t_7^3}\right)^{-1/6}\text.
  \end{equation*}
  We apply \cite[Proposition 7.3]{arXiv:1302.6151} with $r=6$ and use polar coordinates.
\end{proof}

\section{Proof of the main theorem}\label{sec:proof_main}

Let $\alpha(\tS):=\frac{1}{6220800}$ and recall the definitions of
$\omega_\infty(\tS)$ from Theorem~\ref{thm:main} and $\mathcal{R}(B)$ from
\eqref{eq:E6_height_1}--\eqref{eq:E6_height_4}.

\begin{lemma}\label{lem:E6_predicted_volume}
  Define
  \begin{equation*}
    V_0'(B) := \int_{\substack{(\e_1, \ldots, \e_9) \in
        \mathcal{R}(B)\\\abs{\e_1}\text{, }\abs{\e_3}\text{, }\abs{\e_4}\text{,
        }\abs{\e_5}\text{, }\abs{\e_6}\text{, }\abs{\e_7}\geq
        1\\\abs{\e_1^2\e_3^4\e_4^4\e_5^5\e_6^6\e_7^3} \le
        B}}\frac{1}{\abs{\e_4^2\e_5\e_7^3}}\dd \e_1 \cdots \dd \e_9,
  \end{equation*}
  where $\e_1, \ldots, \e_9$ are complex variables. Then
  \begin{equation}\label{eq:E6_predicted_volume}
    \pi^7\alpha(\tS) \omega_\infty(\tS) B(\log B)^6 = 4 V_0'(B).
  \end{equation}
\end{lemma}
\begin{proof}
  Let $\e_1,\e_3,\e_4,\e_5,\e_6,\e_7 \in \CC\smallsetminus\{0\}$, $B>0$, and $l := (B
  \abs{\e_1\e_3^2\e_4^5\e_5^4\e_6^3\e_7^6})^{1/2}$. Let $\e_2, \e_8,
  \e_9$ be complex variables. We apply the coordinate transformation
  $z_0 = l^{-1/3}\cdot \e_9$, $z_1 =
  l^{-1/3}\e_1\e_3^2\e_4^3\e_5^3\e_6^3\e_7^3\cdot \e_2$, $z_2 = l^{-1/3}\cdot \e_1\e_3\e_4\e_5\e_6\e_7\cdot \e_8$
  to $\omega_\infty(\tS)$ and obtain
  \begin{equation}\label{eq:E6_complex_density_torsor}
    \omega_\infty(\tS) = \frac{12}{\pi}\frac{\abs{\e_1\e_3\e_4\e_5\e_6\e_7}}{B}
    \int_{(\e_1, \ldots, \e_9)\in\mathcal{R}(B)}\frac{1}{\abs{\e_4^2\e_5\e_7^3}}\dd \e_2 \dd \e_8 \dd \e_9\text.
  \end{equation}

Since the negative curves $[E_1], \dots, [E_7]$ generate the effective
cone of $\tS$, and $[-K_\tS] = [2E_1+3E_2+4E_3+4E_4+5E_5+6E_6+3E_7]$,
\cite[Lemma~8.1]{arXiv:1302.6151} gives
\begin{equation}\label{eq:E6_alpha}
  \alpha(\tS)(\log B)^6 = \frac{1}{3\pi^6}\int_{\substack{\abs{\e_1}, \abs{\e_3}, \dots, \abs{\e_7} \ge 1\\\abs{\e_1^2\e_3^4\e_4^4\e_5^5\e_6^6\e_7^3}\le B}} \frac{\dd\e_1 \dd\e_3 \dd\e_4 \dd\e_5 \dd\e_6 \dd\e_7}{\abs{\e_1\e_3\e_4\e_5\e_6\e_7}}.
\end{equation}
  The lemma follows by substituting
  \eqref{eq:E6_complex_density_torsor} and \eqref{eq:E6_alpha} in
  \eqref{eq:E6_predicted_volume}.
\end{proof}

To complete the proof of Theorem~\ref{thm:main}, we compare $V_0(B)$ defined in Lemma
\ref{lem:E6_completion} with $V_0'(B)$ defined in Lemma
\ref{lem:E6_predicted_volume}. Starting from $V_0(B)$, we can add the
condition $\abs{\e_1^2\e_3^4\e_4^4\e_5^5\e_6^6\e_7^3} \le B$ and
remove $\abs{\e_2} \ge 1$ with negligible error. Indeed, adding the
condition $\abs{\e_1^2\e_3^4\e_4^4\e_5^5\e_6^6\e_7^3} \le B$ to the
domain of integration for $V_0(B)$ does not change the result. Using
\cite[Lemma 3.5, (3)]{arXiv:1302.6151} applied to
\eqref{eq:E6_height_4} to bound the integral over $\e_8, \e_9$, we see
that $V_0'(B)-V_0(B)$ is
\begin{equation*}
  \ll \int_{\substack{\abs{\e_1},\abs{\e_3}, \dots, \abs{\e_7} \ge 1,\ \abs{\e_2} < 1\\\abs{\e_1^2\e_3^4\e_4^4\e_5^5\e_6^6\e_7^3} \le B}} \frac{B^{5/6}}{\abs{\e_1^4\e_2^3\e_3^2\e_4^2\e_5\e_7^3}^{1/6}}\dd \e_1\cdots\dd \e_7 \ll B(\log B)^5.
\end{equation*}
Using Lemma~\ref{lem:E6_completion} and Lemma~\ref{lem:E6_predicted_volume},
this implies Theorem \ref{thm:main}.

\bibliographystyle{alpha}

\bibliography{manin_dp3_e6_imag_quad}

\end{document}